\documentclass[12pt]{article}
\usepackage[a4paper,margin=1.0in]{geometry}
\usepackage[colorlinks,citecolor=magenta,linkcolor=black]{hyperref}
\pdfpagewidth=\paperwidth \pdfpageheight=\paperheight
\usepackage{amsfonts,amssymb,amsthm,amsmath,eucal,tabu,url}
\usepackage{caption}
\usepackage{float}
\usepackage{pgf}
 \usepackage{array}
 \usepackage{tikz-cd}
 \usepackage{pstricks}
 \usepackage{pstricks-add}
 \usepackage{pgf,tikz}
 \usetikzlibrary{automata}
 \usetikzlibrary{arrows}
 \usepackage{indentfirst}
\usepackage{tabularx} 

\theoremstyle{plain}
\newtheorem{thm}{Theorem}[section]
\newtheorem{theorem}[thm]{Theorem}

\newtheorem{proposition}[thm]{Proposition}
\newtheorem{corollary}[thm]{Corollary}
\newtheorem{conjecture}[thm]{Conjecture}

\theoremstyle{definition}
\newtheorem{definition}[thm]{Definition}
\newtheorem{remark}[thm]{Remark}
\newtheorem{example}[thm]{Example}

\newtheorem{question}[thm]{Question}
\newtheorem{problem}[thm]{Problem}

\newtheorem{thevarthm}[thm]{\varthmname}

\newenvironment{varthm*}[1]{\trivlist\item[]{\bf #1.}\it}{\endtrivlist}

\renewcommand\geq{\geqslant}

\renewcommand\leq{\leqslant}

\newcommand\be{\begin{eqnarray*}}
\newcommand\ee{\end{eqnarray*}}

\newcommand\newop[2]{\def#1{\mathop{\rm #2}\nolimits}}
\newop\log{log}
\newop\ord{ord}
\newop\Gal{Gal}
\newop\SL{SL}
\newop\Bl{Bl}
\newop\mult{mult}
\newop\mass{mass}
\newop\div{div}
\newop\codim{codim}
\newop\sing{sing}
\newop\vdim{vdim}
\newop\edim{edim}
\newop\Ass{Ass}
\newop\size{size}
\newop\reg{reg}
\newop\satdeg{satdeg}
\newop\supp{supp}
\newop\Neg{Neg}
\newop\Nef{Nef}
\newop\Nefh{Nef_H}
\newop\Eff{Eff}
\newop\Zar{Zar}
\newop\MB{MB}
\newop\MBxC{MB\mathit{(x,C)}}
\newop\NnB{NnB}
\newop\Bigg{Big}
\newop\Effbar{\overline{\Eff}}

\def\keywordname{{\bfseries Keywords}}%
\def\keywords#1{\par\addvspace\medskipamount{\rightskip=0pt plus1cm
\def\and{\ifhmode\unskip\nobreak\fi\ $\cdot$
}\noindent\keywordname\enspace\ignorespaces#1\par}}
\def\subclassname{{\bfseries Mathematics Subject Classification
(2020)}\enspace}
\def\subclass#1{\par\addvspace\medskipamount{\rightskip=0pt plus1cm
\def\and{\ifhmode\unskip\nobreak\fi\ $\cdot$
}\noindent\subclassname\ignorespaces#1\par}}

\begin{document}
\title{Singular plane curves: freeness and combinatorics}
\author{Michael Cuntz and Piotr Pokora}
\date{\today}
\maketitle

\thispagestyle{empty}
\begin{abstract}
In this paper we focus on various aspects of singular complex plane curves, mostly in the context of their homological properties and the associated combinatorial structures. We formulate some challenging open problems that can point to new directions in research, for example by introducing weak Ziegler pairs of curve arrangements. Moreover, we construct new examples of different Ziegler pairs, in both the classical and the weak sense, and present new geometric approaches to construction problems of singular plane curves.
\subclass{singular plane curves; minimal free resolutions; combinatorics}
\end{abstract}
\section{Introduction}
The main goal of the present paper is to provide a coherent and short introduction to the world of singular plane curves, which is a subject lying on the boundary of algebra, geometry, and combinatorics. Our idea here is to explore very interesting deep properties of algebraic curves and their importance in the context of very recent big open problems in the so-called \textbf{combinatorial algebraic geometry}. By combinatorial algebraic geometry we mean the field that is based on close bridges between combinatorics and algebraic geometry, and where these fields of research cooperate in full symbiosis. The Leitmotif for our discussion here is the Numerical Terao's Conjecture (we will write \textbf{NTC} for short) and this problem builds the core of the paper. Roughly speaking, \textbf{NTC} predicts that the freeness of reduced plane curves is determined by the weak-combinatorics. This conjecture is somehow surprising and absolutely not obvious. We are aware of the fact that \textbf{NTC} does not hold in the class of line arrangements in the complex projective plane, and we deliver a very detailed discussion regarding this subject, but up to right now there is no single counterexample to this conjecture when we extend our consideration to curves such that irreducible components are no longer, and not only, of degree one. This is a very mysterious thing and we do not know how to explain this phenomenon. However, to get some feeling about this problem, it seems quite natural to explore new classes of plane curves in order to understand which geometric/combinatorial properties can have an impact on the failure of \textbf{NTC}. In the last few years researchers focused exactly on that area of studies, namely to study specific classes of curves that are free and to explore their geometric properties. The first very natural class for such consideration is the class of rational plane curve arrangements with quasi-homogeneous singularities. The main motivation standing behind this choice is very natural. First, it generalizes the framework of line arrangements, and secondly in this class the total Tjurina number of a given curve is determined by its weak-combinatorics.

Let us outline the structure of the present paper. The first part has an expository nature, we focus on the notion of free plane curves. Then we focus on the weak-combinatorics that can be attached to a given reduced plane curve. We will show that this is a very delicate problem and why we cannot use the notion of (strong) combinatorics that is defined for line arrangements verbatim to the world of plane curves. Next, we focus on \textbf{NTC} and we present a detailed outline on that subject emphasizing the recent developments. We conclude our studies by introducing and studying the notion of weak Ziegler pairs for different classes of curves. In particular, we construct the second known example of a weak Ziegler pair consisting of two conic-line arrangements with ordinary quasi-homogeneous singularities. In the last part of the paper, we come back to the classical setting of line arrangement. We construct and describe in detail new examples of Ziegler pairs that are strictly related to a geometric construction revolving around the orchard problem. We should note here that there are only a few examples of Ziegler pairs of line arrangements, and there are no known combinatorial or geometric invariants that can potentially indicate where to look for new cases of Ziegler pairs. More surprisingly, our work suggests that there may exist an infinite sequence of Ziegler pairs of line arrangements that is based on our construction of the orchard problem, and this would be the first such sequence!

We hope that the synergy of the expository part and the current research will be fruitful for the reader and it will help to understand better the current challenges standing behind {\bf NTC}. We should emphasize that there is an excellent recent survey devoted to free plane curves and their geometric properties by Alex Dimca \cite{SurDim}, but our paper is strictly oriented on \textbf{NTC} and weak-combinatorics of plane curves, so we will stay parallel with respect to Alex's survey and we hope that our presentation will play a complementary role.

We work exclusively in the projective setting over the complex numbers. Many computations presented in that survey are performed using \verb}SINGULAR} \cite{Singular}. We will use the local normal forms of plane curve singularities captured from \cite{arnold}.

\section{Freeness of plane curves}
Let $S := \mathbb{C}[x,y,z]$ be the graded ring of polynomials with complex coefficients, and for a homogeneous polynomial $f \in S$ let us denote by $J_{f}$ the Jacobian ideal associated with $f$, that is, the ideal of the form $J_{f} = \langle \partial_{x}\, f, \partial_{y} \, f, \partial_{z} \, f \rangle$. 
We assume in the whole paper that our plane curves $C \, : f=0$ are always reduced. We will need the following definition.
\begin{definition}
Let $p$ be an isolated singularity of a polynomial $f\in \mathbb{C}[x,y]$. Since we can change the local coordinates, assume that $p=(0,0)$.
Furthermore, the number 
$$\mu_{p}=\dim_\mathbb{C}\bigg( \mathbb{C} \{x,y \} / \langle \partial_{x} f, \partial_{y} f \rangle \bigg)$$
is called the (local) Milnor number of $f$ at $p$.

The number
$$\tau_{p}=\dim_\mathbb{C} \bigg( \mathbb{C}\{x, y\} / \langle f, \partial_{x} f, \partial_{y} f \rangle \bigg)$$
is called the (local) Tjurina number of $f$ at $p$.
\end{definition}
\begin{remark}
For a projective situation, with a point $p\in \mathbb{P}^{2}_{\mathbb{C}}$ and a homogeneous polynomial $F\in \mathbb{C}[x,y,z]$, we can take local affine coordinates such that $p=(0,0,1)$, and then the dehomogenization of $F$.\\
Finally, the total Tjurina number of a given reduced curve $C \subset \mathbb{P}^{2}_{\mathbb{C}}$ is defined as
$$\tau(C) = \sum_{p \in {\rm Sing}(C)} \tau_{p}.$$ 
\end{remark}
We will work mostly with quasi-homogeneous singularities.
\begin{definition}
A singularity is called quasi-homogeneous if and only if there exists a holomorphic change of variables so that the defining equation becomes weighted homogeneous.
\end{definition}
 More specifically, $f(x,y) = \sum_{i,j}c_{i,j}x^{i}y^{j}$ is weighted homogeneous if there exist rational numbers $\alpha, \beta$ such that $\sum_{i,j} c_{i,j}x^{i\cdot \alpha} y^{j \cdot \beta}$ is homogeneous. One can show that if $f(x,y)$ is a convergent power series with an isolated singularity at the origin, then $f(x,y)$ is in the ideal generated by the partial derivatives if and only if $f$ is quasi-homogeneous. It means that in the quasi-homogeneous case one has $\tau_{p} = \mu_{p}$, i.e., the local Tjurina number of $p$ is equal to the local Milnor number of $p$. Moreover, if $C \, : f=0$ is a reduced plane curve with only quasi-homogeneous singularities, then
$$\tau(C) = \sum_{p \in {\rm Sing}(C)} \tau_{p} = \sum_{p \in {\rm Sing}(C)} \mu_{p} = \mu(C),$$ which means that the total Tjurina number of $C$ is equal to the total Milnor number of $C$.

Before we define the ultimate class of curves in our investigations, i.e., free plane curves, we introduce the following general definition.
 
\begin{definition}
\label{hom}
Let $C : f=0$ be a reduced curve in $\mathbb{P}^{2}_{\mathbb{C}}$ of degree $d$ given by $f \in S$. Denote by $M(f) := S/ J_{f}$ the Milnor algebra. We say that $C$ is $m$-syzygy when $M(f)$ has the following minimal graded free resolution:
$$0 \rightarrow \bigoplus_{i=1}^{m-2}S(-e_{i}) \rightarrow \bigoplus_{i=1}^{m}S(1-d - d_{i}) \rightarrow S^{3}(1-d)\rightarrow S$$
with $e_{1} \leq e_{2} \leq ... \leq e_{m-2}$ and $1\leq d_{1} \leq ... \leq d_{m}$. 
\end{definition}
\begin{definition}
The $m$-tuple $(d_{1}, ..., d_{m})$ in Definition \ref{hom} is called the exponents of $C$.
\end{definition}
In the light of Definition \ref{hom}, being $m$-syzygy curve is a homological condition that is decoded by the shape of the minimal free resolution of the Milnor algebra. Among the $m$-syzygy curves, the most important ones, from our very subjective point of view, are the following \cite{Saito}.
\begin{definition}
We say that a reduced curve $C \subset \mathbb{P}^{2}_{\mathbb{C}}$ of degree $d$ is free if and only if $C$ is $2$-syzygy, and in that case $d_{1}+d_{2}=d-1$.
\end{definition}
In other words, a reduced plane curve $C \subset \mathbb{P}^{2}_{\mathbb{C}}$ is free if the corresponding minimal free resolution of the Milnor algebra is Hilbert-Burch. Now we want to present a geometric approach showing that the freeness of plane curves is decoded by their total Tjurina numbers and one special exponent.
\begin{definition}
Let $C \subset \mathbb{P}^{2}_{\mathbb{C}}$ be a reduced curve given by $f \in S_{d}$.
We define the graded $S$-module of algebraic relations associated with $f$ as
$${\rm AR}(f) = \{(a,b,c)\in S^{\oplus 3} : a \cdot\partial_{x}f+ b\cdot \partial_{y} f + c\cdot \partial_{z} f =0 \}.$$
Then the minimal degree of Jacobian relations among the partial derivatives of $f$ is defined as
$${\rm mdr}(f):=\min\{r : \, {\rm AR}(f)_{r} \neq 0 \}.$$
\end{definition}
\begin{remark}
In the light of Definition \ref{hom}, we have 
$${\rm mdr}(f) := d_{1}.$$
\end{remark}
\begin{remark}
Sometimes we also write ${\rm mdr}(C) = {\rm mdr}(f)$, where $C : f=0$.
\end{remark}
\begin{example}
Let us consider the arrangement $\mathcal{L} \subset \mathbb{P}^{2}_{\mathbb{C}}$ given by 
$$Q(x,y,z)=xyz.$$
Using the language of the elementary projective geometry, $\mathcal{L}$ is the fundamental triangle consising of $3$ nodes as the intersection points located at the three fundamental points of $\mathbb{P}^{2}_{\mathbb{C}}$. 
Obviously ${\rm mdr}(Q)>0$ since there are no $\alpha,\beta,\gamma \in \mathbb{C}\setminus \{0\}$ such that
$$\alpha \partial_{x}f + \beta\partial_{y} f + \gamma\partial_{z}f = \alpha yz + \beta xz + \gamma xy = 0.$$
On the other hand, it is easy to see that ${\rm mdr}(Q)=1$, i.e., we have the following relation
$$x \partial_{x}f - y\partial_{y} f = 0.$$
\end{example}
As we have seen so far, to show that a given reduced plane curve $C \, : f=0$ is free, we need to compute the minimal free resolution of $M(f)$, which can be difficult without a reasonable computer assistance. However, there is an interesting result, which can be considered as folklore and that might be successfully used in some cases. This result is known in the literature as Saito's criterion \cite[Theorem 8.1]{Dimca}.
\begin{theorem}
Let $C : f=0$ be a reduced curve in $\mathbb{P}^{2}_{\mathbb{C}}$ of degree $d$. Let  $r_{1},r_{2} \in {\rm AR}(f)$ be two non-trivial syzygies of the form
$r_{i} = (f_{1}^{i},f_{2}^{i},f_{3}^{i})$ with $i \in \{1,2\}$. Then
$C$ is free if and only if
$${\rm Det}\begin{pmatrix}
x & f_{1}^{1} & f_{1}^{2} \\
y & f_{2}^{1} & f_{2}^{2} \\
z & f_{3}^{1} & f_{3}^{2}
\end{pmatrix} = c \cdot f,$$
where $c$ is a non-zero constant.
\end{theorem}
It is worth recalling that the above Saito's criterion holds in a very general setting, i.e., we can formulate it even for reduced hypersurfaces $V \subset \mathbb{P}^{N}_{\mathbb{C}}$, and it is a quite effective tool provided that we see appropriate candidates for syzygies $r_{1}, r_{2}$.
\begin{example}
Let us come back to our canonical example of the fundamental triangle $\mathcal{L} : xyz=0$. Observe that we have the following two natural syzygies:
$$r_{1} = (x,-y,0), \quad r_{2} = (x,0,-z).$$
Observe that
$${\rm Det}\begin{pmatrix}
x & x & x \\
y & -y & 0 \\
z & 0 & -z
\end{pmatrix} = 3xyz,$$
hence $\mathcal{L}$ is free with exponents $(d_{1},d_{2})=({\rm mdr}(Q),d-1-{\rm mdr}(Q)) = (1,1)$.
\end{example}
As we have said, Saito's criterion is an effective tool provided that we can find two suitable syzygies. On the other hand, the freeness of reduced algebraic curves can be checked via the following deep result by du Plessis and Wall \cite{duP}.

\begin{theorem}[Freeness criterion]
\label{dp}
Let $C \, : f=0$ be a reduced curve in $\mathbb{P}^{2}_{\mathbb{C}}$ of degree $d$. Then the curve $C$ with $d_{1}={\rm mdr}(f)\leq (d-1)/2$ is free if and only if
\begin{equation}
\label{duPles}
(d-1)^{2} - r(d-r-1) = \tau(C).
\end{equation}
\end{theorem}
 The reader can easily use the above criterion to verify, once again, that the fundamental triangle $\mathcal{L} : xyz=0$ is free.

Finally, let us define the second most important class of curves in our considerations, namely nearly-free curves.
\begin{definition}
A reduced plane curve $C \subset \mathbb{P}^{2}_{\mathbb{C}}$ of degree $d$ is nearly-free if $C$ is $3$-syzygy with $d_{1}+d_{2}=d$ and $d_{2}=d_{3}$.
\end{definition}
This definition is somehow complicated to handle, but we have a very useful result due to Dimca \cite{Dimca1}, which allows us to check the nearly-freeness using the information about the total Tjurina number and the minimal degree of the Jacobian relations.
\begin{theorem}
A reduced curve $C : f=0$ in $\mathbb{P}^{2}_{\mathbb{C}}$ of degree $d$ with $d_{1} = {\rm mdr}(f)$ is nearly free if and only if
\begin{equation}
(d-1)^{2}-d_{1}(d-d_{1}-1)=\tau(C)+1.
\end{equation}   
\end{theorem}
\section{Weak-combinatorics versus strong combinatorics and \textbf{NTC}}
In this section we elaborate about the notion of the weak-combinatorics and strong combinatorics that can be attached to a reduced plane curve. We start with the strong combinatorics that is defined for line arrangements.
\begin{definition}
Let $\mathcal{L} \subset \mathbb{P}^{2}_{\mathbb{C}}$ be an arrangement of $d$ lines. Then by the intersection lattice $L(\mathcal{L})$ we mean the set of all flats, i.e., non-empty intersections of (sub)families of lines in $\mathcal{L}$ with the order defined by the reverse inclusion, i.e., $X \leq Y$ if and only if $Y \subset X$.
\end{definition}
It turns out that the information delivered by the intersection lattice of a given line arrangement $\mathcal{L}$ can be equivalently decoded by the so-called Levi graph. Here we take a step forward and define the notion of the Levi graph for arrangements consisting of smooth curves and admitting arbitrary singularities.
\begin{definition}
Let $\mathcal{C} = \{C_{1}, ..., C_{k}\} \subset \mathbb{P}^{2}_{\mathbb{C}}$ be a reduced curve consisting of $k$ smooth components, each of degree ${\rm deg} \, C_{i} = d_{i} \geq 1$. Then the \emph{Levi graph} $G = (V,E)$ is a bipartite graph with  $V : = V_{1} \cup V_{2} = \{x_{1}, ..., x_{s}, y_{1}, ..., y_{k}\}$, where each vertex $y_{i}$ corresponds to a curve $C_{i}$, each vertex $x_{j}$ corresponds to an intersection point $p_{j} \in {\rm Sing}(\mathcal{C})$ and $x_{j}$ is joined with $y_{i}$ by an edge in $E$ if and only if $p_{j}$ is incident with $C_{i}$.  
\end{definition} 
The above definition is a rather weak notion if we study arbitrary arrangements of smooth plane curves. Consider the following example.
\begin{example}
Let $\mathcal{L}$ be an arrangement of two lines intersecting along a single point, and let $\mathcal{C}$ be an arrangement of two smooth conics intersecting only along one point (which is of course possible). Then we can easily construct the associated Levi graphs, and in both cases these graphs have the same shape, namely
$$G=(\{x_{1}, y_{1}, y_{2}\}, \{\{x_{1},y_{1}\}, \{x_{1},y_{2}\}\}).$$
\end{example}
Obviously the Levi graph does not deliver any information about the types of singularities of a given curve $\mathcal{C}$.
\begin{example}
Let $\mathcal{C} = \{C_{1}, C_{2}\}$ be an arrangement of two smooth conics such that the singular locus of $\mathcal{C}$ consists of three points, two simple nodes $p_{1},p_{2}$ (i.e. $A_{1}$ singularities), and one tacnode $p_{3}$ (i.e. $A_{3}$ singularity). Let $x_{i}$ be the vertex corresponding to point $p_{i}$ with $i \in \{1,2,3\}$. Then the associated Levi graphs has the following form:
$$G=(\{x_{1},x_{2},x_{3}, y_{1}, y_{2}\}, \{\{x_{1},y_{1}\}, \{x_{1},y_{2}\},\{x_{2},y_{1}\},\{x_{2},y_{2}\},\{x_{3},y_{1}\}, \{x_{3},y_{2}\}\}).$$
Obviously, just by looking at $G$ and without any additional data, we cannot distinguish nodes from tacnodes.
\end{example}
Now let us pass to the following fundamental definition for this paper.

\begin{definition}
Let $C = \{C_{1}, ..., C_{k}\} \subset \mathbb{P}^{2}_{\mathbb{C}}$ be a reduced curve such that each irreducible component $C_{i}$ is \textbf{smooth}. The weak-combinatorics of $C$ is a vector of the form $(d_{1}, ..., d_{s}; m_{1}, ..., m_{p})$, where $d_{i}$ denotes the number of irreducible components of $C$ of degree $i$, and $m_{j}$ denotes the number of singular points of a curve $C$ of a given type $M_{j}$.
\end{definition}
In the above definition we refer to types of singularities, which means that for us these types are determined by their local normal forms, e.g. \cite{arnold}.
\begin{example}
Let $\mathcal{C} = \{C_{1}, C_{2}\}$ be an arrangement of two smooth conics such that the singular locus of $\mathcal{C}$ consists of three points, namely two simple nodes and one tacnode. Denote by $n_{2}$ the number of nodes and by $t_{3}$ the number of tacnodes. Then the weak-combinatorics of $\mathcal{C}$ has the following form
$$(d_{2};n_{2},t_{3}) = (2;2,1).$$
\end{example}
Why do we introduce these various notions of combinatorics that can be attached to a given reduced curve $C$? It is all about Terao's freeness conjecture and $\textbf{NTC}$.
We start with the classical Terao's conjecture devoted to line arrangements \cite{terao}.
\begin{conjecture}[Terao]
Let $\mathcal{A},\mathcal{B} \subset \mathbb{P}^{2}_{\mathbb{C}}$ be two line arrangements such that their intersection lattices $L(\mathcal{A}), L(\mathcal{B})$ are isomorphic. Assume that $\mathcal{A}$ is free, then $\mathcal{B}$ must be free.
\end{conjecture}
Notice that Terao's freeness conjecture is widely open -- we know that it holds with up to $14$ lines \cite{BarKuh}, which might be slightly disappointing. However, classical Terao's conjecture is very demanding and in order to understand this problem well one needs to understand the geometry of moduli spaces of line arrangements that are free. It is very unclear whether the mentioned conjecture holds in general, and based on that problem Dimca and Sticlaru in \cite{DimcaSticlaru} defined the class of nearly-free curves. Shortly afterwards it turned out that nearly-free curves might be crucial for Terao's conjecture. It is believed that if there exists a counterexample to Terao's freeness conjecture, then we should be able to find two line arrangements with isomorphic intersection lattices, where one arrangement is free and the second arrangement is nearly-free. Being aware of such fundamental problems, we can wonder whether there is a natural way to extend the setting designed for Terao's freeness conjecture to higher degree curves, for instance to the case of conic-line arrangements. We have the following example that comes from \cite{STo}.
\begin{example}[Counterexample to a naive Terao's conjecture]
Consider the following conic-line arrangement
$$\mathcal{CL}_{1} \, : \, xy\cdot(y^{2} + xz)\cdot(y^{2}+x^{2} + 2xz)=0.$$
Observe that the intersection point $P = (0:0:1)$ has multiplicity $4$. Moreover, it is quasi-homogeneous, but not \textbf{ordinary}. Using \verb}SINGULAR}, we can check that $P$ is a singular point of type $Z_{1,0}$ with $\tau_{P}=\mu_{P}=15$. One can show that $\mathcal{CL}_{1}$ is free with exponents $(2,3)$. If we perturb a bit line $y=0$, taking for instance $x-13y = 0$, we get a new conic-line arrangement 
$$\mathcal{CL}_{2} \, : \, x\cdot(x-13y)\cdot(y^{2} + xz)\cdot(y^{2}+x^{2} + 2xz)=0.$$
In this new arrangement, the intersection point $P=(0:0:1)$ has multiplicity $4$, but it is not longer quasi-homogeneous, and $\mathcal{CL}_{2}$ is not free. In fact, the arrangement $\mathcal{CL}_{2}$ is nearly-free with exponents $(d_{1},d_{2},d_{3}) = (3,3,3)$. Obviously the Levi graphs of $\mathcal{CL}_{1}$, $\mathcal{CL}_{2}$ are isomorphic.
\end{example}
The example presented above is crucial. Firstly, we see that (strong) combinatorics for curve arrangements, i.e., where the irreducible components are not just lines, does not determine the freeness, so the classical Terao's conjecture cannot be naively extended. On the other hand, the above example shows that the notion of weak-combinatorics might allow us to distinguish arrangements in the context of the freeness property. From this perspective, it is natural to ask whether the following conjecture can hold.

\begin{conjecture}[Numerical Terao's Conjecture]
Let $C_{1}, C_{2}$ be two reduced curves in $\mathbb{P}^{2}_{\mathbb{C}}$ with the property that they possess only smooth irreducible components and they admit only quasi-homogeneous singularities. Suppose that $C_{1}$ and $C_{2}$ have the same weak-combinatorics and let $C_{1}$ be free, then $C_{2}$ is also free.
\end{conjecture}
The assumption that we consider only quasi-homogeneous singularities is technical and not intuitive at first sight. \textbf{If we assume that our curves admit only smooth irreducible components and quasi-homogeneous singularities, then the total Tjurina number is determined by the weak-combinatorics}. For instance, it is well-known that ordinary intersection points for curve arrangements of multiplicity $m>4$ are, in general, not quasi-homogeneous, see \cite[Example 4.2]{STo} or \cite[Exercise 7.31]{RCS}. The assumption that all singularities are quasi-homogeneous allows us, for instance, to compute the local Tjurina number of such ordinary singular points using only Milnor's formula involving the number of branches and the intersection indices.

In the forthcoming two sections we collect all recent developments regarding \textbf{NTC}. First of all, we start our discussion in the setting of line arrangements, and then we jump into the scenario of conic-line arrangements in the complex plane.

\subsection{\textbf{NTC} versus triangular line arrangements}

This section is based on \cite{Mar}. Here we want to show that \textbf{NTC} is false in the class of line arrangements in the complex plane. In order to construct the mentioned counterexample, we need to introduce a special class of line arrangements, the so-called triangular arrangements in the complex projective plane.
\begin{definition}
An arrangement $\mathcal{L}$ of $abc$ lines is called \textbf{triangular} if $\mathcal{L}$ is given by an equation of type
$$Q(x,y,z) = xyz \prod_{i=1}^{a-1}(x-\alpha_{i}y)\cdot  \prod_{j=1}^{b-1}(y-\beta_{j}z) \cdot  \prod_{i=k}^{c-1}(x-\gamma_{k}z)=0,$$
where $\alpha_{i},\beta_{j},\gamma_{k} \in \mathbb{C}$.
\end{definition}
The above definition is motivated by the well-known example of a supersolvable line arrangement, called sometimes the full $n$th CEVA arrangement, or the full monomial arrangement, and we denote it as $\mathcal{FMA}_{n}$. This arrangement is given by the following equation depending on $n\geq 2$, namely
$$F_{n}(x,y,z) =  xyz(x^{n}-y^{n})(y^{n}-z^{n})(x^{n}-z^{n})=0.$$
Here we want to present the following result that is proved in \cite[Theorem 6.2]{Mar}.
\begin{theorem}
\label{pair}
There exists a pair of line arrangements in the complex projective plane, each of which has weak-combinatorics of the form
$$(d_{1};n_{2},n_{3},n_{4},n_{5},n_{6}) = (15; 24 ,12, 0, 0, 3),$$
where $n_{i}$ denotes the number of (ordinary) $i$-fold intersection points, such that one is free and the second is nearly-free.
\end{theorem}
\begin{proof}
The first line arrangement $\mathcal{L}_{1}$ is constructed by removing from the full monomial arrangement $\mathcal{FMA}_{6}$ defined by
$$F_{6}(x,y,z) = xyz(x^{6}-y^{6})(y^{6}-z^{6})(x^{6}-z^{6}) =0$$
the following lines
$$\ell_{1} : x-z=0, \quad \ell_{2} : x-ez=0, \quad \ell_{3} : y-z=0, $$
$$\ell_{4} : y-ez=0, \quad \ell_{1} : x-e^{2}y=0, \quad \ell_{6} : x-e^{4}y=0,$$
where $e^{2}-e+1=0$.
We can check, using \verb}SINGULAR}, that $\mathcal{L}_{1}$ is free with exponents $(d_{1},d_{2}) = (7,7)$. 

The second line arrangement $\mathcal{L}_{2}$ is constructed as follows. Take the full monomial arrangement $\mathcal{FMA}_{5}$ given by 
$$F_{5}(x,y,z) = xyz(x^{5}-y^{5})(y^{5}-z^{5})(x^{5}-z^{5}) =0$$
and remove the following three lines
$$\ell_{1} : x-z = 0, \quad \ell_{2} : x-y = 0, \quad \ell_{2} : y-z = 0.$$
We can check that $\mathcal{L}_{2}$ is not free, but only nearly-free with exponents $(d_{1},d_{2},d_{3}) = (6,9,9)$.
\end{proof}
From the above result we can extract the following crucial observations.
\begin{corollary}
The Numerical Terao's conjecture does not hold in the class of triangular line arrangements in the complex projective plane.
\end{corollary}
\begin{corollary}
In the class of line arrangements in the complex projective plane, the minimal degree of non-trivial Jacobian relations is not determined by the weak-combinatorics!
\end{corollary}
\subsection{\textbf{NTC} holds for conic-line arrangements with nodes, tacnodes, and ordinary triple points}
In this section we report results devoted to the freeness of conic-line arrangements with simple singularities. More precisely, $\mathcal{CL} = \{\ell_{1}, ...,\ell_{d}, C_{1}, ..., C_{k}\} \subset \mathbb{P}^{2}_{\mathbb{C}}$ be an arrangement of $d\geq 0$ lines and $k\geq 0$ smooth conics admitting only the following quasi-homogeneous singularities (here we use local normal forms from \cite{arnold}):
\begin{center}
\begin{tabular}{ll}
$A_{k}$ with $k\geq 1$ & $: \, x^{2}+y^{k+1}  = 0$, \\
$D_{4}$ & $: \, y^{2}x + x^{3}  = 0$,\\
$X_{9}$ with $a^{2}\neq 4$ & $: \, x^{4} + y^{4} +ax^{2}y^{2}= 0$. 
\end{tabular}
\end{center}
In particular, $A_{1}$ singularities are just nodes, $D_{4}$ singularities are ordinary triple points, and $X_{9}$ singularities are ordinary quadruple points.

Our first result is devoted to the case when we have only smooth conics, see \cite[Theorem A]{Pokora}.
\begin{theorem}
There does not exist any arrangement of $k\geq 2$ smooth conics with $A_{1}, A_{3}, D_{4}, X_{9}$ singularities that is free.
\end{theorem}
This result is optimal in the sense that if we add to the above list of singularities type $A_{5}$, then we can find an arrangement that is free, see \cite[Remark 2.5]{P}. On the other hand, this result stands in odds with the picture of line arrangements where we can find many free arrangements admitting nodes, triple and quadruple points.

Now we can focus on the case of conic-line arrangements $\mathcal{CL} \subset \mathbb{P}^{2}_{\mathbb{C}}$ with $d\geq 1$ and $k\geq 1$. We assume that our arrangements admit only singularities of type $A_{1}, A_{3}, D_{4}$ -- this is the first non-trivial situation that distinguishes the geometry of conic arrangements and line arrangements.
Our main result, obtained in a joint paper with A. Dimca \cite{DimcaPokora}, delivers a complete classification of such arrangements.

\begin{theorem}
\label{thm11}
Let $\mathcal{CL}$ be an arrangement of $d \geq 1$ lines and $k \geq 1$ smooth conics having only nodes, tacnodes, and  ordinary triple points as singularities. Then  $\mathcal{CL}$ is free if and only if one of the following cases occur:
\begin{enumerate}

\item[(1)] $d=k=1$ and $\mathcal{CL}$ consists of a smooth conic and a tangent line, which means we can only have one tacnode.

\item[(2)] $d=2$, $k=1$ and $\mathcal{CL}$ consists of a smooth conic and two tangent lines, so in this case we have one node and two tacnodes.

\item[(3)] $d=3$, $k=1$ and either $\mathcal{CL}$ is a smooth conic inscribed in a triangle, or $\mathcal{CL}$ is a smooth conic circumscribed in a triangle.
In the first case we have three nodes and three tacnodes, and in the second case we have only three ordinary triple points as intersections.

\item[(4)] $d=3$, $k=2$ and $\mathcal{CL}$ consists of a triangle $\Delta$, a smooth conic inscribed in $\Delta$, and another smooth conic circumscribed in $\Delta$. In this case we have $5$ tacnodes and and three ordinary triple points as singularities.

\end{enumerate}
In particular, a free conic-line arrangement having only nodes, tacnodes, and ordinary triple points is determined up to a projective equivalence by the weak-combinatorics.
\end{theorem}

This theorem implies the following classification result.
\begin{corollary}
\label{cor5}
Numerical Terao's Conjecture holds for conic-line arrangements with nodes, tacnodes, and ordinary triple points.
\end{corollary}
The following remark explains the situation when $k=0$, so we arrive at the scenario of line arrangements with double and triple points as intersections.
\begin{remark}
Numerical Terao's Conjecture holds for line arrangements having only double and triple points as singularities and it follows from the following observations. First of all, a free line arrangement $\mathcal{A}: f=0$ has to satisfy the condition that $d = \deg f \leq 9$, see \cite{Kabat} for an elementary argument. Next, observe that the freeness condition implies that
$$d_{1} + d_{2} = d - 1 \quad \text{ with } d_{1} \leq d_{2},$$
and we have 
$(d-1)/2 \geq d_1 = {\rm mdr}(f)$. This gives us $d_{1} \leq 3$ or
$d_1= {\rm mdr}(f) = 4$ and $d=9$. Note that if $\mathcal{A}':f'=0$ has the same weak-combinatorics as $\mathcal{A} :f=0$, then $\tau(\mathcal{A})=\tau(\mathcal{A}')$ since the total Tjurina number of a reduced curve with quasi-homogeneous singularities is determined by the weak-combinatorics. It implies, in particular, that
$r'= {\rm mdr}(f') \leq {\rm mdr}(f)$, and this follows from the maximality of the Tjurina number for free reduced curves, see \cite{duP}.
In the first case, using the complete classification of line arrangements with ${\rm mdr}(f) \leq 3$ provided in \cite{BT}, we can conclude the statement. In the second case, we use again the maximality of the Tjurina number of free curves \cite{duP} and we conclude that $\tau(\mathcal{A})=\tau(\mathcal{A}')=48$, and this maximal value is obtained when the number of nodes is $0$ and the number of triple points is equal to $12$. The only line arrangement with these invariants is the dual Hesse arrangement of lines which is free with exponents $(d_{1},d_{2}) = (4,4)$, and this arrangement is unique up to the projective equivalence.
\end{remark}
\section{Weak Ziegler pairs}
This short section is devoted to the intriguing notion of Ziegler pairs. Let us recall that this topic is strictly motivated by the following question.
\begin{question}
Is it true that for a given complex line arrangement the minimal free resolution of the associated Milnor algebra is determined by its intersection lattice?
\end{question}
At the first glance we have some doubts since there is no reasonable argument supporting the claim that the shape of the resolution of the Milnor algebras associated with line arrangements can be determined by the combinatorial structure. From our perspective, and from the perspective of \textbf{NTC}, we can focus our attention on the minimal degree of non-trivial Jacobian relations as a crucial invariant of the whole resolution. This idea leads us to the following object.
\begin{definition}[Ziegler pair]
We say that two line arrangements $\mathcal{L}_{1}, \mathcal{L}_{2} \subset \mathbb{P}^{2}_{\mathbb{C}}$ form a Ziegler pair if the intersection lattices of $\mathcal{L}_{1}$ and $\mathcal{L}_{2}$ are isomorphic, but ${\rm mdr}(\mathcal{L}_{1}) \neq {\rm mdr}(\mathcal{L}_{2})$.
\end{definition}
As we can see, Ziegler pairs have a strong connection with possible counterexamples to Terao's freeness conjecture. Recall that for line arrangements the total Tjurina number is determined by the intersection lattice (since all singularities are quasi-homogeneous), so if we could find a counterexample to Terao's conjecture, it would be all about the minimal degrees of non-trivial Jacobian relations being different.
The first pair of line arrangements with the same intersection lattices but different minimal degrees of non-trivial Jacobian relations was found by Ziegler \cite{Ziegler}, and we will say a few words about this construction. This pair consists of two arrangements having exactly $9$ lines with $6$ triple and $18$ double points as intersections, but their geometries are different. In the first case, all six triple points are on a conic, but in the second case, only $5$ triple points are on a conic, and one point is off the conic. Geometrically speaking, the condition that $6$ points are on a conic is unexpected, and this is a crucial feature for this example. This construction has recently been revisited and studied in detail by Dimca and Sticlaru \cite{DimSt}.
It is worth noticing that in the literature we can also find another variant of Ziegler pairs that we would like to present now.
\begin{definition}[Ziegler pair, 2nd variant]
We say that two line arrangements $\mathcal{L}_{1}, \mathcal{L}_{2} \subset \mathbb{P}^{2}_{\mathbb{C}}$ form a Ziegler pair if the intersection lattices of $\mathcal{L}_{1}$ and $\mathcal{L}_{2}$ are isomorphic, but they have different ${\rm AR}$ modules of algebraic relations.
\end{definition}
However, the notion of Ziegler pairs is not suitable for curve arrangements, in its whole generality, since we can detect some pathological situations.
\begin{example}
Consider an arrangement of lines $\mathcal{L}$ given by 
$$Q(x,y,z)=xy.$$
Clearly ${\rm mdr}(Q)=0$.
Now we consider the arrangement of conics $\mathcal{C}$ given by
$$G(x,y,z)=(x^2-yz)(x^{2}+z^{2}-yz).$$
Obviously these two different curve arrangements have the same Levi graphs and both are free, but
$$0={\rm mdr}(Q)\neq {\rm mdr}(G)=1.$$
\end{example}
In order to avoid such situations, we define the following.
\begin{definition}[Weak Ziegler pair]
Consider two reduced plane curves $C_{1}, C_{2} \subset \mathbb{P}^{2}_{\mathbb{C}}$ such that all irreducible components of $C_{1}$ and $C_{2}$ are smooth. We say that a pair $(C_{1},C_{2})$ forms a weak Ziegler pair if $C_{1}$ and $C_{2}$ have the same weak-combinatorics, but they have different minimal degrees of non-trivial Jacobian relations, i.e., ${\rm mdr}(C_{1}) \neq {\rm mdr}(C_{2})$.
\end{definition}
\begin{proposition}
    The line arrangements constructed in Theorem \ref{pair} form a weak Ziegler pair.
\end{proposition}
\begin{proof}
The arrangements $\mathcal{L}_{1}$ and $\mathcal{L}_{2}$ have the same weak-combinatorics, but
$$7 = {\rm mdr}(\mathcal{L}_{1}) \neq {\rm mdr}(\mathcal{L}_{2})=6,$$
so they form a weak Ziegler pair.
\end{proof}
It is natural to wonder whether we can construct new weak Ziegler pairs in different classes of curve arrangements, for instance in the class of conic-line arrangements. Let us recall the following example by Schenck and Toh\u aneanu from \cite{STo} providing the first weak Ziegler pair for conic-line arrangements with ordinary but not quasi-homogeneous singularities.
\begin{example}
\label{2bb}
 Let us consider the arrangement $\mathcal{C}_{1}$ being the union of the following five smooth conics:
$$\begin{array}{*{3}c}
C_1& : &(x-3z)^2+(y-4z)^2-25z^2=0 \\
C_2& : &(x-4z)^2+(y-3z)^2-25z^2=0 \\
C_3& : &(x+3z)^2+(y-4z)^2-25z^2=0 \\
C_4& : &(x+4z)^2+(y-3z)^2-25z^2=0 \\
C_5& : &(x-5z)^2+y^2-25z^2=0.
\end{array}$$
The arrangement $\mathcal{C}_{1}$ has $13$ ordinary singular points, $10$ of these
points are nodes, while at the
points $(0:0:1),(1:i:0),(1:-i:0)$ all the five conics meet.
We can check that the quintuple point $q=(0:0:1)$ is not quasi-homogeneous since $15 = \tau_{q} \neq \mu_{q}=16$.
If we add the following lines
$$\ell_{1} : z=0, \quad \ell_{2} : x-iy = 0, \quad \ell_{3} : x+iy =0,$$
then we obtain arrangement $\mathcal{CL}_{1} = \{C_{1}, ..., C_{5}, \ell_{1},\ell_{2},\ell_{3}\}$ which is free with exponents $(d_{1}, d_{2})=(6,6)$.

Next, let us consider the arrangement $\mathcal{C}_{2}$ being the union of the following five smooth conics:
$$\begin{array}{ccl}
C_1'& : &x^2+8y^2+21xy-xz-8yz=0 \\
C_2'& : &x^2+5y^2+13xy-xz-5yz=0 \\
C_3'& : &x^2+9y^2-4xy-xz-9yz=0 \\
C_4'& : &x^2+11y^2+xy-xz-11yz=0 \\
C_5'& : &x^2+17y^2-5xy-xz-17yz=0.
\end{array}$$
Observe that $\mathcal{C}_{2}$ is combinatorially and weak-combinatorially identical to $\mathcal{C}_{1}$, but the quintuple points $(0:0:1),(1:0:1),(0:1:1)$, where all
the branches meet, are not quasi-homogeneous since  $15=\tau \neq \mu=16$.
If we add the lines 
$$\ell_{1}' : x=0, \quad \ell_{2}' : y=0, \quad \ell_{3}' : x+y-z=0,$$
then we obtain arrangement $\mathcal{CL}_{2} = \{C_{1}', ..., C_{5}', \ell_{1}',\ell_{2}',\ell_{3}'\}$ having the same strong combinatorics as $\mathcal{CL}_{1}$, and since all singularities are ordinary, then the arrangements have the same weak-combinatorics, but $\mathcal{CL}_{2}$ is not free. In fact, $\mathcal{CL}_{2}$ is just $4$-syzygy with exponents $(d_{1},d_{2},d_{3},d_{4}) = (7,7,7,7)$.
\end{example}
\begin{remark}
In the above example we can literally see the phenomenon that if singularities of our curve are not quasi-homogeneous, then the total Tjurina number is not determined by the weak-combinatorics.
\end{remark}
Now we would like to introduce, similarly to the case of line arrangements, the second variant on the notion of Weak Ziegler pairs.

\begin{definition}[Weak Ziegler pair, 2nd variant]
Consider two reduced plane curves $C_{1}, C_{2} \subset \mathbb{P}^{2}_{\mathbb{C}}$ such that all irreducible components of $C_{1}$ and $C_{2}$ are smooth. We say that a pair $(C_{1},C_{2})$ forms a weak Ziegler pair if $C_{1}$ and $C_{2}$ have the same weak-combinatorics, but they have different ${\rm AR}$ modules of algebraic relations.
\end{definition}
Here we would like to present an interesting example of a Weak Ziegler pair (according to the 2nd variant) formed by using a pair of conic-line arrangements with quasi-homogeneous singularities. This pair is the second known example in the class of conic-line arrangements with quasi-homogeneous ordinary singularities (cf. \cite{LBW}), and it is also interesting due to another feature that we will explain in a separate remark.
\begin{example}
\label{WZZ}
Let us consider two conic-line arrangements $\mathcal{CL}_{1}, \mathcal{CL}_{2} \subset \mathbb{P}^{2}_{\mathbb{C}}$ defined as follows:
\begin{equation*}
\mathcal{CL}_{1} \, : (-24x^{2}-23y^{2}+76yz+195z^{2})(y-3x-5z)(y+3x-5z)(y+z)(y-3z)x(x+y+z) = 0,
\end{equation*}
\begin{equation*}
\mathcal{CL}_{2} \, : (-24x^{2}-23y^{2}+76yz+195z^{2})(y-2x-3z)(y+2x-3z)(y+z)(y-3z)x(x+y+z) = 0.
\end{equation*}
Both arrangements have only \textbf{ordinary quasi-homogeneous} singularities, namely they have $12$ nodes, $3$ ordinary triple points, and one ordinary quadruple point. The figures below show their geometric realizations.
\vskip 0.3cm
\noindent
\begin{minipage}{0.45\textwidth}
\centering
\begin{tikzpicture}[scale=0.45,line cap=round,line join=round,x=1.0cm,y=1.0cm]
\clip(-6.979656031286098,-7.476422086706328) rectangle (7.166401239887506,9.511309733564092);
\draw [rotate around={90.:(0.,1.6521739130434783)},line width=0.5 pt] (0.,1.6521739130434783) ellipse (3.3478260869565215cm and 3.2773376434211823cm);
\draw [line width=.5 pt,domain=-9.979656031286098:11.166401239887506] plot(\x,{(--5.--3.*\x)/1.});
\draw [line width=0.5 pt,domain=-9.979656031286098:11.166401239887506] plot(\x,{(--5.-3.*\x)/1.});
\draw [line width=.5 pt,domain=-9.979656031286098:11.166401239887506] plot(\x,{(-1.-0.*\x)/1.});
\draw [line width=.5 pt,domain=-9.979656031286098:11.166401239887506] plot(\x,{(--3.-0.*\x)/1.});
\draw [line width=.5 pt] (0.,-5.476422086706328) -- (0.,9.511309733564092);
\draw [line width=.5 pt,domain=-9.979656031286098:11.166401239887506] plot(\x,{(-1.-1.*\x)/1.});
\end{tikzpicture}

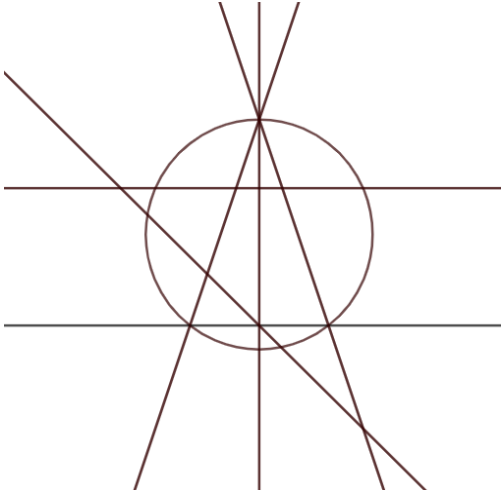
\captionof{figure}{Arrangement $\mathcal{CL}_{1}$.}
\end{minipage}
\hfill
\begin{minipage}{0.45\textwidth}
\centering
\begin{tikzpicture}[scale=0.45,line cap=round,line join=round,x=1.0cm,y=1.0cm]
\clip(-6.979656031286098,-7.476422086706328) rectangle (7.166401239887506,9.511309733564092);
\draw [rotate around={90.:(0.,1.6521739130434783)},line width=.5 pt] (0.,1.6521739130434783) ellipse (3.3478260869565215cm and 3.2773376434211823cm);
\draw [line width=.5 pt,domain=-9.868360993016763:11.27769627815684] plot(\x,{(--3.--2.*\x)/1.});
\draw [line width=.5 pt,domain=-9.868360993016763:11.27769627815684] plot(\x,{(--3.-2.*\x)/1.});
\draw [line width=.5 pt,domain=-9.868360993016763:11.27769627815684] plot(\x,{(-1.-0.*\x)/1.});
\draw [line width=.5 pt,domain=-9.868360993016763:11.27769627815684] plot(\x,{(--3.-0.*\x)/1.});
\draw [line width=.5 pt] (0.,-6.47807743113034) -- (0.,9.103227926576531);
\draw [line width=.5 pt,domain=-9.868360993016763:11.27769627815684] plot(\x,{(-1.-1.*\x)/1.});
\end{tikzpicture}

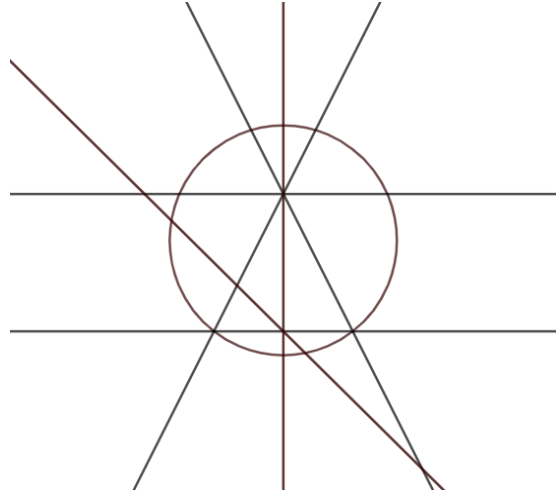
\captionof{figure}{Arrangement $\mathcal{CL}_{2}$.}
\end{minipage}     
\end{example}
\noindent
Denote by $Q_{i}$ the defining equation for the arrangement $\mathcal{CL}_{i}$ with $i \in\{1,2\}$. Now we look at the ${\rm AR}$ modules of algebraic relations. We have the following resolutions:
\begin{equation*}
0\rightarrow S(-7) \oplus S(-6) \rightarrow S(-6)\oplus S(-5)^{\oplus2} \oplus S(-4) \rightarrow AR(Q_{1}),
\end{equation*}
\begin{equation*}
0\rightarrow S(-7) \rightarrow S(-5)^{\oplus2} \oplus S(-4) \rightarrow AR(Q_{2}),
\end{equation*}
which means that our arrangements have different ${\rm AR}$ modules with the same weak-combinatorics, and thus form a Weak Ziegler pair.
\begin{remark}
\label{rrr}
In \cite{her}, Pokora with Dimca and Abe introduced a new homological approach to plane curves. Let $C: f=0$ be a reduced plane curve of degree $d$ and let $d_1$ and $d_2$ be the minimal degrees of a minimal system of generators for the module ${\rm AR}(f)$. We say that $C$ is of type $k$ if $d_1+d_2=d-1+k$. Given this definition, our arrangements in Example \ref{WZZ} are curves of type $2$. Moreover, according to \cite[Proposition 1.13]{her}, the arrangement $\mathcal{CL}_{1}$ is of type $2B$ and the arrangement $\mathcal{CL}_{2}$ is of type $2A$. Finally, let us point out also that the conic-line arrangement $\mathcal{CL}_{2}$ from Example \ref{2bb} is a curve of type $2B$.
\end{remark}

We finish our considerations by proposing the following problem.
\begin{problem}
Construct more examples of weak Ziegler pairs in the class of curve arrangements (so arrangements with not only lines as irreducible components) admitting only quasi-homogeneous singularities.
\end{problem}
\section{Ziegler pairs and the orchard problem}
In this section we would like to present a new type of Ziegler pairs, which is somehow surprisingly related to a classical construction in the orchard problem. Let us use the notation captured from \cite{Burr}. We define a $(p,t)$-arrangement consisting of $p$ points and $t$ lines in the real projective plane, chosen so that each line has exactly $3$ points on it. The classical orchard problem asks to find an arrangement with the largest $t$ for any given value of $p$. Using the duality in projective spaces, we can reformulate this problem, which now boils down to finding line arrangement in the real projective plane with the largest number of triple intersections $t^{*}$ for a given number of lines $p^{*}$. This problem, in the light of a recent paper by Green and Tao \cite{Green} and by Yuzvinsky \cite{Y}, uses a very nice geometric construction involving cubic curves. We would like to present here a combinatorial description that allows us to construct an arrangement of $n$ lines.

Consider the cyclic group $\mathbb{Z}/n\mathbb{Z}$ of order $n$. The lines in the real projective plane are labeled by
its elements, namely $0,1,...,n-1$, and now we require that three lines $a,b,c$ meet in a point if and
only if $a+b+c = 0$. The resulting line arrangement, denoted by $\mathcal{L}_{n}$, has exactly $n$ lines, and by \cite[Proposition 2.6]{Green} it has exactly $\lfloor n(n-3)/6\rfloor + 1$ triple intersection points. This description corresponds to the picture where we have a group of points $E_{n}$ in an elliptic curve with the natural group action. 

Let us now focus on the case $n=10$. We can compute that the line arrangement $\mathcal{L}_{10}$ can be defined by the following equation:
\begin{multline*}
Q_{(s,t)} = xyz(tx+sy+tz)(y+z)(x+y+z)(x+sy+tz)\bigg((-s+t+1)x+(-s+t)z\bigg) \\ \bigg(x+(s-t)y\bigg) \bigg((-st^2+st+t^3-t)x+(-st+t^2)y+(-st^2+t^3)z\bigg),    
\end{multline*}
subject to the conditions that $st - 2s - t^2 + t + 1 =0$ and $(s,t) \not\in \{{ (1/2,0), (1,1), (0,w) }\}$, with $w$ being a root of $f(r) = r^{2}-r-1$. Geometrically speaking, our arrangement has the $2$-parameter realization space, which is a hyperbola minus a bunch of points - these are the points at which the arrangement degenerates. We can check that there is an integral point on this hyperbola, namely $(s,t)=(5,3)$, and this point is going to be crucial for us right now.

Consider the arrangement $\mathcal{L}_{10}^{1}$ given by $Q_{(5,3)}=0$. Using \verb}SINGULAR} we can directly compute the resolution of $AR(Q_{(5,3)})$, namely 
\begin{equation*}
0\rightarrow S(-8) \rightarrow S(-6)^{\oplus2} \oplus S(-5) \rightarrow AR(Q_{(5,3)}),
\end{equation*}
and, in the light of Remark \ref{rrr}, the arrangement $\mathcal{L}_{10}^{1}$ is of type $2A$. Let us now consider arrangement $\mathcal{L}^{2}_{10}$ given by $Q_{((5\sqrt{5}+15)/4, \sqrt{5}+3)} = 0$. Again, using \verb}SINGULAR} we can directly compute the resolution of $AR$ module obtaining the following:
\begin{equation*}
0\rightarrow S(-8) \oplus S(-7) \rightarrow S(-7) \oplus S(-6)^{\oplus2} \oplus S(-5) \rightarrow AR(Q_{(5\sqrt{5}+15)/4,\sqrt{5}+3)}),
\end{equation*}
and, in the light of Remark \ref{rrr}, the arrangement $\mathcal{L}_{10}^{2}$ is of type $2B$. 

Summing up, we have two line arrangements $\mathcal{L}_{10}^{1}, \mathcal{L}_{10}^{2}$ having the same combinatorics but different $AR$ modules, so we have just constructed a new Ziegler pair.

Let us now focus on the case $n=12$. We have the following equation for $\mathcal{L}_{12}$:
\begin{multline*}
Q^{'}_{(s,t)} =  xyz(x+y+z)(x+y+tz)(tx+sy+tz)(y+z)(x + (s-t)y)\bigg((-s+t+1)x + (-s+t)z)\bigg) \\
(x+sy+tz)\bigg((-s t^2 + st + t^3 - t)x+(-st + t^2)y+(-st^2 + t^3)z\bigg) \\
\bigg((s^2 t - st^2 - 2st + t^2 + t)x+(s^3t - s^3 - 2s^2 t^2 + s^2 t + s^2 + s t^3 - st)y + (s^2 t - 2st^2 + t^3)z\bigg),
\end{multline*}
subject to the condition that $$ s = t(2-t), \quad t \notin \{ 0,1,2 \} $$
and, additionally, $t$ is not a root of $r^2-r+1$ or $r^2-2r+2$.

Consider the arrangement $\mathcal{L}_{12}^{1}$ given by $Q'_{(-3,3)} = 0$. Using \verb}SINGULAR} we can directly compute the resolution of $AR(Q^{'}_{(-3,3)})$, namely
\begin{equation*}
0\rightarrow S(-10) \rightarrow S(-8) \oplus S(-7) \oplus S(-6) \rightarrow AR(Q^{'}_{(-3,3)}),
\end{equation*}
which means that $\mathcal{L}_{12}^{1}$ is of type $2A$. 
\noindent
Let us now consider the arrangement $\mathcal{L}_{12}^{2}$ given by $Q'_{(s',t')} = 0$ with $s'=(1 - \sqrt3 ) (\sqrt3 - 3)$ and $t'=\sqrt{3}-1$. Using \verb}SINGULAR} we can directly compute the resolution of $AR$ module obtaining 
\begin{equation*}
0\rightarrow S(-10) \oplus S(-9) \rightarrow S(-9) \oplus S(-8) \oplus S(-7) \oplus S(-6) \rightarrow AR(Q_{(s',t')}),
\end{equation*}
which means that $\mathcal{L}_{12}^{2}$ is of type $2B$. It means that we have found another Ziegler pair. Based on our experiments we can formulate the following problem.
\begin{problem}
Is it true that the realization space of the arrangement $\mathcal{L}_{n}$ for every even integer $n\geq 10$ always admits two geometric realizations $\mathcal{L}_{n}^{1}$, $\mathcal{L}_{n}^{2}$ such that they form a Ziegler pair with one arrangement being of type $2A$ and the second arrangement of type $2B$? 
\end{problem}
\section*{Acknowledgement}
Piotr Pokora wants to thank Alex Dimca for all discussions regarding \textbf{NTC} and Ziegler pairs. Both authors would like to thank anonymous referees for useful comments.

Piotr Pokora is supported by the National Science Centre (Poland) Sonata Bis Grant \textbf{2023/50/E/ST1/00025}. For the purpose of Open Access, the author has applied a CC-BY public copyright license to any Author Accepted Manuscript (AAM) version arising from this submission.

Michael Cuntz's research was conducted within the framework of the DFG-funded Research Training Group RTG 2965: From Geometry to Numbers.

\vskip 0.5 cm

Michael Cuntz,
Institut f\"ur Algebra, Zahlentheorie und Diskrete Mathematik, Fakultat f\"ur Mathematik und Physik, Leibniz Universit\"at Hannover, Welfengarten 1, D-30167 Hannover, Germany.\\
\nopagebreak
\textit{E-mail address:} \texttt{cuntz@math.uni-hannover.de}
\vskip 0.5 cm

Piotr Pokora,
Department of Mathematics,
University of the National Education Commission Krakow,
Podchor\c a\.zych 2,
PL-30-084 Krak\'ow, Poland. \\
\nopagebreak
\textit{E-mail address:} \texttt{piotr.pokora@uken.krakow.pl}
\end{document}